\documentclass[12pt]{article}
\usepackage{amsmath,amssymb,latexsym}
\usepackage{enumerate}
\usepackage{amsthm}
\usepackage{graphicx}
\usepackage{epsfig}
\usepackage{xcolor}

\newcommand\NN{\mathrm{I\!N}}

\newcommand\CC{\mathbb{C}}

\newcommand{\rI}{\mathrm{I}}
\newcommand{\lcm}{\mathrm{lcm}}

\newtheorem{theorem}{Theorem}[section]
\newtheorem{definition}[theorem]{Definition}
\newtheorem{corollary}[theorem]{Corollary}
\newtheorem{example}[theorem]{Example}
\newtheorem{lemma}[theorem]{Lemma}
\newtheorem{remark}[theorem]{Remark}

\newcommand{\address}{Address: Department of Mathematics, University of North Texas, 1155 Union Circle \#311430, Denton, TX 76203-5017, USA; E-mail: kiko.kawamura@unt.edu, tobeymathis@my.unt.edu}

\numberwithin{equation}{section}

\title{$\Delta$-revolving sequences and \\ self-similar sets in the plane \\ 
{\small - Dedicated to the late Professor Shunji Ito -}}

\author{Kiko Kawamura and Tobey Mathis \\University of North Texas \footnote{\address}}

\begin{document}
\maketitle

\begin{abstract}
Initiated by Mizutani and Ito's work in 1987, Kawamura and Allen recently showed that certain self-similar sets generalized by two similar contractions have a natural complex power series representation, which is parametrized by past-dependent revolving sequences. 

In this paper, we generalize the work of Kawamura and Allen to include a wider collection of self-similar sets. We show that certain self-similar sets consisting of more than two similar contractions also have a natural complex power series representation, which is parametrized 
by {\it $\Delta$-revolving sequences}. This result applies to several other famous self-similar sets such as the Heighway dragon, Twindragon, and Fudgeflake. 
\end{abstract}



\section{Introduction}

The history of systematic mathematical research on self-similar fractals dates back to 1981, when Hutchinson proved the following celebrated theorem~\cite{Hutchinson-1981}: For any finite family of similar contractions $\{\psi_{0},\psi_{1},\dots,\psi_{m-1}\}$ on $\CC$, which is called an iterated function system or IFS, there exists a unique non-empty compact solution $E$ of the set equation:
$$E=\bigcup_{i=0}^{m-1}\psi_{i}(E).$$ 
We call $E$ an attractor or a self-similar set for the IFS. 
\medskip

It is well known that any point of $E$ can be represented by at least one coding sequence $(x_i)_{i=1}^{\infty}$ such that 
\begin{equation}
\label{eq:coding}
E=\left\{ \lim_{n \to \infty}\psi_{x_1} \circ \psi_{x_2} \circ \cdots \circ \psi_{x_n}(0): (x_i)_{i=1}^{\infty} \in \{0, 1, 2, \cdots, m-1\}^{\NN} \right\}.
\end{equation}
For more details, see Falconer's book~\cite{Falconer-2003}. 
\medskip

Eleven years prior to Hutchinson’s paper \cite{Hutchinson-1981}, C.~Davis and D.~E.~Knuth~\cite{Davis+Knuth-1970} in 1970 introduced the idea that some tiling fractals are associated with the complex number system. This idea was further discussed by Gilbert~\cite{Gilbert-1982}. In \cite{Davis+Knuth-1970}, Davis and Knuth introduced the concept of {\it revolving sequences} to represent Gaussian integers:
for any $z=x+iy$ with $x, y \in \mathbb{Z}$,  there exists a finite revolving sequence with length $N$: $(\delta_{0}, \delta_{1},\dots \delta_{N})$ such that  
\begin{equation*}
  z=\sum_{n=0}^{N} \delta_{N-n}(1+i)^{n}, 
\end{equation*}
where $\delta_j \in \{0, 1, -1, i, -i\}$ with the restriction that the non-zero values must follow the cyclic pattern from left to right:
$$\cdots \to 1 \to (-i) \to (-1) \to i \to 1 \to \cdots.$$ 
Notice that a revolving sequence is created by a cycle of 90 degree clockwise rotation on the unit circle. 
\medskip

Let $W$ be the set of all revolving sequences. M.~Mizutani and S.~Ito in 1987~\cite{Ito-1987} considered the following set of points in the complex plane:
\begin{equation}
\label{Mizutani-Ito-X}
  X:=\left\{ \sum_{n=1}^{\infty} \delta_{n}(1+i)^{-n}: (\delta_{1}, \delta_{2}, \dots ) \in W \right\}.
\end{equation}
Using techniques from symbolic dynamics, they proved that $X$ (which they called {\it Tetradragon}) is tiled by four rotated paper-folding dragons. Paper-folding dragon is generated by the IFS: 
\begin{equation}
\label{rev-dragon}
\begin{cases}
 \psi_0(z)=(\frac{1-i}{2})z,& \\
 \psi_1(z)=(\frac{-1-i}{2})z + \frac{1-i}{2}.&
\end{cases}
\end{equation}
Observe that revolving sequences $(\delta_n)_{n=1}^{\infty}$ are past-dependent sequences.

In the same paper, they mentioned an interesting conjecture: suppose $\delta_n$ moves on the unit circle counterclockwise instead of clockwise, then 
\begin{equation*}
X^{R}:=\left\{ \sum_{n=1}^{\infty} \overline{\delta_{n}}(1+i)^{-n}: (\delta_{1},\delta_{2}, \delta_{3}, \dots )\in W \right\} 
\end{equation*}
is a union of four L\'evy's dragon curves. L\'evy's dragon is a well-known continuous curve with positive area. Mizutani and Ito's computer simulation of $X^{R}$ empirically confirmed the conjecture; however, they could not give a mathematical proof. Kawamura in 2002~\cite{Kawamura-2002} proved that their conjecture is correct using a different approach from the viewpoint of functional equations.

These two concrete examples indicate that not only tiling fractals but more generally some self-similar attractors can be parametrized by past-dependent sequences. 

Recently, Kawamura and Allen~\cite{Kawamura-Allen-2020} defined {\it Generalized Revolving Sequences (GRS)}, where the 90 degree angle of rotation is replaced with a more general angle $|\theta|=\frac{2\pi q}{p} \leq \pi$. In other words, $\delta_{k} \in \{0, 1, e^{i\theta}, e^{2i\theta}, \cdots, e^{(p-1)i \theta}\}$ with the restriction that the non-zero values must follow the following cyclic pattern from left to right:
$$\cdots \to 1 \to e^{i\theta} \to e^{2i\theta} \to e^{3i\theta} \to \cdots.$$

Define $W_{\theta}$ to be the set of all generalized infinite revolving sequences with parameter $\theta$. Kawamura and Allen considered the following set. 
\begin{equation}
\label{eq:X1}
X_{\alpha,\theta}:=\left\{ \sum_{n=1}^{\infty} \delta_{n}\alpha^{n}: (\delta_{1}, \delta_{2},\dots )\in W_{\theta} \right\}.
\end{equation}
where $\alpha \in\CC$ such that $|\alpha|<1$ and $|\theta|=\frac{2 \pi q}{p}\leq\pi$. Note that generalized revolving sequences are past-dependent. 
 
Kawamura and Allen proved that $X_{\alpha, \theta}$ is a union of $p$ rotated copies of the self-similar attractor $M_{\alpha, \theta}$ of an IFS consisting of two similarities with the same scale factor, one of which involves a rotation through an angle $\theta$:
\begin{equation}
\label{eq:IFS1} 
\begin{cases}
 \psi_0(z)=\alpha z,& \\
 \psi_1(z)=(\alpha e^{i \theta}) z+\alpha.
\end{cases}
\end{equation}

\begin{figure}
\begin{center}
\includegraphics[height=4cm,width=4cm, bb=100 100 1000 1000]{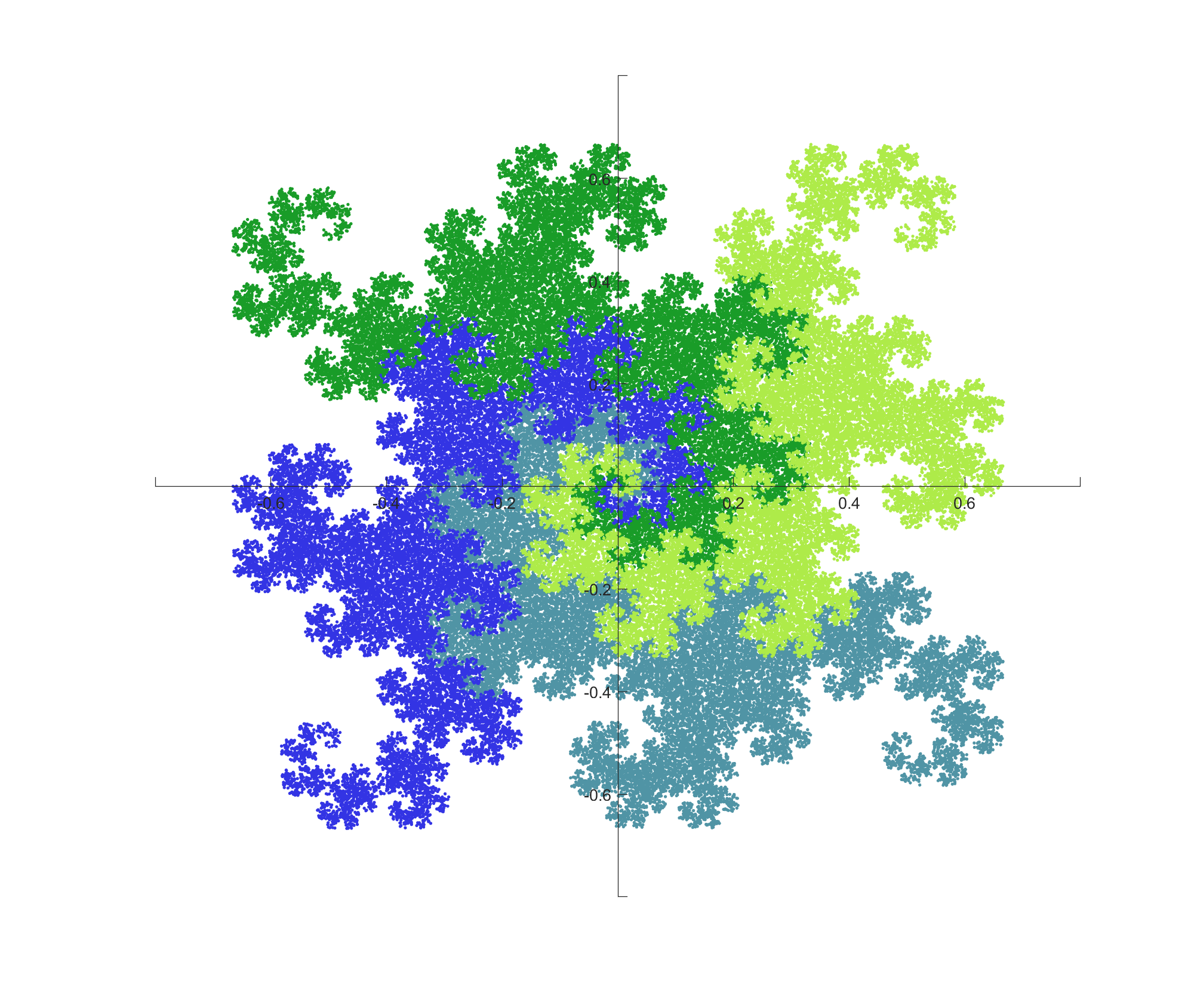}
\qquad \qquad
\includegraphics[height=4cm,width=4cm, bb=100 100 1000 1000]{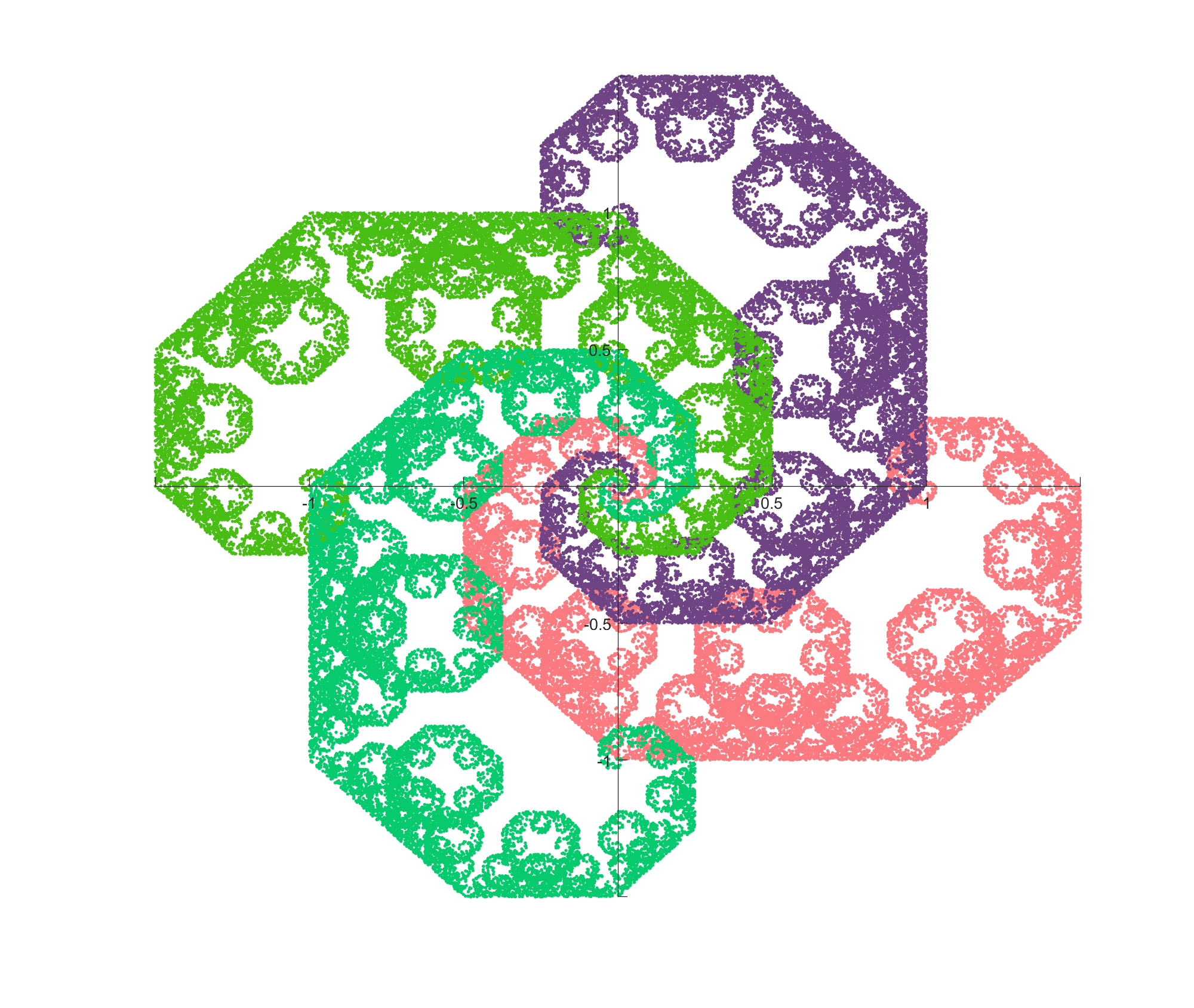}
\end{center}
\caption{$X_{ \alpha,\theta}$ for $\alpha=\frac{1-i}{2}, \theta=-\frac{\pi}{2}$ (left), $\alpha=\frac{1-i}{2}, \theta=\frac{\pi}{2}$ (right)}
\label{figure1}
\end{figure}

Figure \ref{figure1} shows two examples of $X_{\alpha, \theta}$. Observe that this is a generalization of Mizutani-Ito's and Kawamura's   results: if $\alpha=(1-i)/2$ and $\theta=-\pi/2$ then $X_{ \alpha,\theta}$ is a union of four paper-folding dragons~\cite{Ito-1987}. If $\alpha=(1-i)/2$ and $\theta=\pi/2$ then $X_{ \alpha,\theta}$ is a union of four L\'evy dragon curves.
\medskip

In this paper, we continue building on the work initiated by Mizutani and Ito to include a wider collection of self-similar sets. In particular, the following questions are discussed.
\begin{enumerate} 
\item Does replacing the constant term $\alpha$ in \eqref{eq:IFS1} by an arbitrary constant give any major influence to the property of generalized revolving sequences in \eqref{eq:X1}?
\item Consider self-similar sets generated by the IFS consisting of more than two similar contractions. Can we find a complex power series representation, which is parametrized by past-dependent sequences? How does the number of contractions influence the properties of generalized revolving sequences?
\end{enumerate}

\medskip
In section \ref{delta-revolving sequences}, we introduce {\it $\Delta$-revolving sequences} allowing revolutions on the unit circle by more than one angle, while keeping the current-dependency. We illustrate the difference between generalized revolving sequences and $\Delta$-revolving sequences with examples. 

\medskip

In section \ref{main results}, we show that certain self-similar attractors of IFSs consisting of $m \geq 2$ similar contractions have a natural complex power series representation, which is parametrized by $\Delta$-revolving sequences. This main theorem applies to several other famous self-similar sets such as the Heighway dragon, Twindragon, and Fudgeflake. 

\medskip

Lastly in section \ref{remark}, we introduce a slight modification of $\Delta$-revolving sequences called {\it $\Delta_0$-revolving sequences}. $\Delta_0$-revolving sequences are a more natural extension of the notion of generalized revolving sequences than $\Delta$-revolving sequences. As a corollary of the main theorem, we give a condition when certain self-similar attractors can also be parametrized by $\Delta_0$-revolving sequences.

\section{$\Delta$-Revolving Sequences}
\label{delta-revolving sequences}

First, we review some notations and a result from \cite{Kawamura-Allen-2020}.
\begin{definition}[Revolving Angle]
        We say $\theta\in(-\pi,\pi]$ is a $\textbf{revolving angle}$ if $\theta$ is a nonzero rational multiple of $2\pi$, that is, $|\theta|=(2\pi q)/p$ where $p\in\mathbb{N},\:q\in\mathbb{N}_0$.
    \end{definition}

\begin{definition}[Generalized Revolving Sequence]
Define 
$$\Delta_{\theta}:=\{0, 1, e^{i \theta}, e^{2i \theta}, \cdots e^{(p-1)i \theta}\},$$
where $\theta$ is a revolving angle. We say that a sequence $(\delta_n)\in \Delta_{\theta}^{\NN}$ satisfies the \textbf{Generalized Revolving Condition (GRC)}, if        
        \begin{enumerate}
            \item $\delta_1$ is free to choose. 
						\item If $\delta_1=\delta_2=\cdots=\delta_n=0$, then $\delta_{n+1}$ is free to choose.
            \item Otherwise, $\delta_{n+1}=0$ or $\delta_{n+1}=\delta_{j_0(n)} e^{i\theta}$, where $j_0(n)=\max \{j \leq n: \delta_j \neq 0 \}$.
				\end{enumerate}		
\end{definition}

\begin{example}
        If $\theta=\pi/3$, $\Delta_{\theta}:=\{0, 1, e^{i\pi/3}, e^{2i\pi/3}, e^{i\pi}, e^{4i\pi/3}, e^{5i\pi/3}\}.$ 
				\begin{equation}
				\label{example:GRS}
				(\delta_n)=(0, 1, e^{i(\pi/3)}, 0, e^{i(2\pi/3)}, 0, 0, e^{i\pi}, e^{i(4\pi/3)}, 0, \dots)
				\end{equation}
				is a generalized revolving sequence. 
\end{example}

Define $W_{\theta}$ as the set of all generalized revolving sequences with parameter $\theta$: 
\begin{equation*}
W_{\theta}:=\{(\delta_1, \delta_2, \cdots) \in \Delta_{\theta}^{\NN}: (\delta_1, \delta_2, \cdots) \mbox{ satisfies the GRC}\}.
\end{equation*}

Observe that $\delta_n$ moves on the unit circle counterclockwise if $\theta >0$, and clockwise if $\theta < 0$. Roughly speaking, $j_0(n)$ is the last time before $n$ that $\delta_j$ is on the unit circle. Therefore, generalized revolving sequences are past-dependent. 

\begin{theorem}[Kawamura-Allen, 2020]
\label{theorem-Kawamura-Allen}
For a given $\alpha \in\CC$ such that $|\alpha|<1$ and $\theta$ is a revolving angle, define
\begin{equation*}
X_{\alpha,\theta}:=\left\{ \sum_{n=1}^{\infty} \delta_{n}\alpha^{n}: (\delta_n)\in W_{\theta} \right\}.
\end{equation*}
Then $X_{\alpha, \theta}$ is a union of $p$ rotated copies of $M_{\alpha,\theta}$:
\begin{equation*}
X_{\alpha,\theta}=\bigcup_{l=0}^{p-1}(e^{i \theta})^l M_{\alpha,\theta},
\end{equation*}
where $M_{\alpha,\theta}=\psi_0(M_{\alpha,\theta})\cup\psi_1(M_{\alpha,\theta})$ is the self-similar set generated by the IFS
\begin{equation*}
\begin{cases}
 \psi_0(z)=\alpha z,& \\
 \psi_1(z)=(\alpha e^{i \theta}) z+\alpha.
\end{cases}
\end{equation*}
\end{theorem}
\medskip

Next, we loosen the restriction of a generalized revolving sequence, allowing revolutions on the unit circle by more than one angle.

    \begin{definition}[Generator Set]
        Let $S:=\{\theta_0, \theta_1,\dots,\theta_{m-1}\}$ where $\theta_0, \theta_1,\dots,\theta_{m-1}$ are distinct angles and $\theta_0=0$.  Its elements can be written as $\theta_j=\frac{2\pi q_j}{p_j}$.  
				We refer to the set S as a \textbf{generator set}.
    \end{definition}
        
    \begin{definition}[Revolving Group Generated by $S$]
        Let $S=\{\theta_0, \theta_1,\dots,\theta_{m-1}\}$ be a generator set. Define 
				$$\Delta:=\big{\{}e^{i\sum_{n=1}^{m-1}k_n\theta_n}\::\:k_n\in\{0,1,\dots,p_n-1\}\big{\}}.$$
		\end{definition}
				
		\begin{example}
        If $S=\{0,\pi,\frac{2\pi}{3}\}$, then  
				$$\Delta= \{1,e^{i(\pi)},e^{i(\frac{2\pi}{3})},e^{i(\frac{4\pi}{3})},e^{i(\pi+\frac{2\pi}{3})},e^{i(\pi+\frac{4\pi}{3})}\}.$$
		\end{example}
		From this example, it is clear that $\Delta$ is a group, which we call the \textbf{revolving group generated by S}. Notice that $\Delta$ does not contain $0$, and $|\Delta|=\lcm (|p_1|, \cdots, |p_{m-1}|)$. 

\begin{definition}[$\Delta$-Revolving Sequence]
		We say that a sequence $(\gamma_n)\in\Delta^\mathbb{N}$ satisfies the \textbf{$\Delta$-Revolving Condition (DRC)}, if
        \begin{enumerate}
            \item $\gamma_1$ is free to choose in $\Delta$,
            \item For $k \geq 1$, $\gamma_{k+1}=\gamma_{k} e^{i\theta_j }$ for some $\theta_j \in S$.
        \end{enumerate}
        If $(\gamma_n)$ satisfies the DRC, we call $(\gamma_n)$ a \textbf{$\Delta$-Revolving Sequence}.
		\end{definition}
		
		\begin{example}
        If $S=\{0,\pi,\frac{2\pi}{3}\}$,  $\Delta= \{1,e^{i(\pi)},e^{i(\frac{2\pi}{3})},e^{i(\frac{4\pi}{3})},e^{i(\pi+\frac{2\pi}{3})},e^{i(\pi+\frac{4\pi}{3})}\}$. And, for examples,  
				\begin{equation}
				\label{example:delta}
				(\gamma_n)=(1, e^{i\pi},e^{i\pi},e^{i(\pi+\frac{2\pi}{3})},e^{i(\pi+\frac{2\pi}{3})},e^{i(\pi+\frac{2\pi}{3})}, e^{i(\frac{2\pi}{3})}, \dots)
				\end{equation}
				is a $\Delta$-revolving sequence. 
				\end{example}
				
Define $W^{\Delta}$ as the set of all $\Delta$-revolving sequences with generator set $S$ and $W_{1}^{\Delta}$ be the subset of $W^{\Delta}$ with the restriction: $\gamma_1=1$.

Compare examples: \eqref{example:GRS} and \eqref{example:delta}. It is clear that the generalized revolving condition is associated with a single angle of rotation on the unit circle, while the $\Delta$-revolving condition is associated with multiple angles of rotation on the unit circle. Observe also that generalized revolving sequences allow that $\delta_n$ goes back to the origin and stays there as long as it wants, while $\Delta$-revolving sequences do not allow that $\gamma_n$ goes back to the origin; however, they allow "staying in place". Informally, each digit $\gamma_n$ can choose freely either to stay ($\theta_0=0$) or move ($\theta_j \in S \setminus \{\theta_0\}$). However, the next digit $\gamma_{k+1}$ is determined based on the current position $\gamma_k$. Therefore, $\Delta$-revolving sequences are current-dependent, while the generalized revolving sequences are past-dependent. 

Alternatively, due to the group structure of $\Delta$, a $\Delta$-Revolving Sequence can be thought of as an infinite walk on the Cayley Digraph of the group $\Delta$ generated by the set with elements corresponding to $S$. 

    \begin{figure}
    \begin{center}
		\includegraphics[height=4.5cm,width=4.5cm, bb=50 50 400 400]{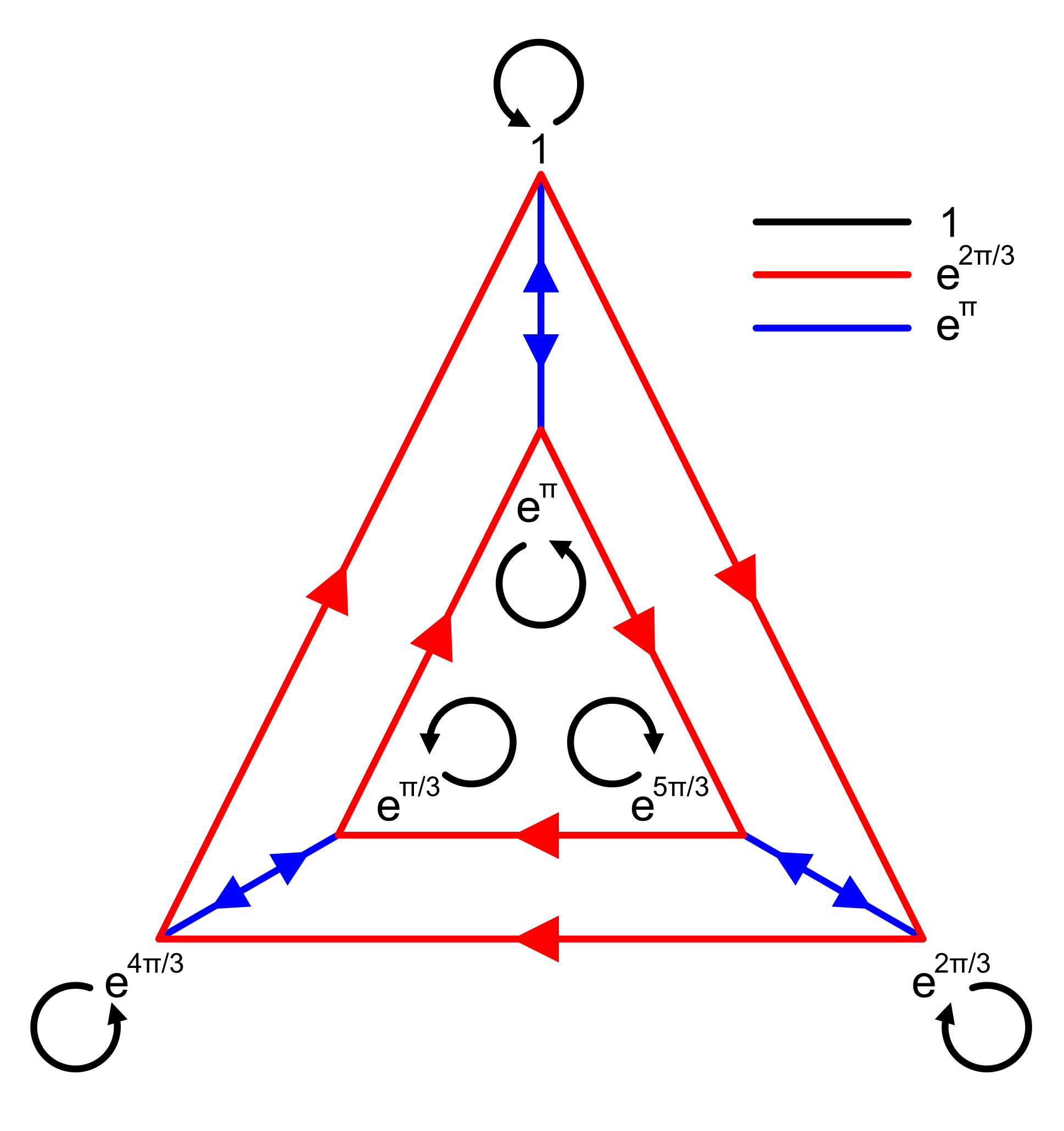} 
		\end{center}
		\caption{Cayley Digraph for generator set $S=\{0,\frac{2\pi}{3},\pi\}$}
    \label{figure4} 
    \end{figure}

\section{Main results}
\label{main results}

\begin{lemma}
		\label{lemma:coding expression} 
    Consider a standard IFS $\{\psi_0,\psi_1,\dots,\psi_{m-1}\}$ given by 
		\begin{equation}
        \label{IFS-general}
                \psi_k(z)=\alpha_k z+ c_k, 
        \end{equation}
				where $\alpha_{k} \in\mathbb{C}$ such that $|\alpha_k|<1$ and $c_k \in\mathbb{C}$ for $k=0,1,\dots,m-1$. 
				
		Then the self-similar set $E=\cup_{i=0}^{m-1}\psi_{i}(E)$ generated by \eqref{IFS-general} has the following natural series representation:
    \begin{equation}
    \label{Attractor}
        E=\left\{\sum_{n=1}^\infty c_{x_n}\prod_{k=0}^{m-1}\alpha_k^{I_k(x,n-1)}\::\: x=(x_n)\in\{0,1,\dots,m-1\}^\mathbb{N}\right\},
    \end{equation}
    where $I_k(x,n):=\#\{j\leq n\::\:x_j=k\}$ for $k=0,1,\dots m-1$ and $n \in \NN$.
		\end{lemma}
		Notice that $I_k(x,n)$ is the number of $k$'s occurring in the first $n$ digits of the sequence $(x_n)$. 
		
		\begin{proof}
		From \eqref{eq:coding}, it follows that every point in $E$ can be represented by at least one coding sequence $(x_n)$ given in \eqref{Attractor}. Vice versa, the set on the right of \eqref{Attractor} satisfies the set equation $E=\cup_{i=0}^{m-1}\psi_{i}(E)$. Let $M:=\{0,1,\dots,m-1\}^\NN$. It is well known that $M$ is compact in the product topology. Define the function $f: M \to \CC$ by 
		 \begin{equation*}
    f(x)=\sum_{n=1}^\infty c_{x_n}\prod_{k=0}^{m-1}\alpha_k^{I_k(x,n-1)}.
    \end{equation*}
		Since it is easily seen that $f$ is continuous and $f(M)=E$, it follows that $E$ is compact so it is closed.
		\end{proof}

Before presenting the main theorem, let us define one more notation. 

\begin{definition}[Constant Sequence of $(\gamma_n)$] 

Let $S=\{\theta_0, \theta_1,\dots,\theta_{m-1}\}$ be a generator set and $\Delta$ be the revolving group generated by S. 

For $\Sigma:=(\gamma_n)\in W^\Delta$ and $c_k \in \mathbb{C}$ for $k=0, 1, \dots m-1$, define the function $s_\Sigma:\mathbb{N}\longrightarrow \{c_0,c_1,\dots,c_{m-1}\}$ by 
        \begin{equation*}
            s_\Sigma(n)= c_{k}, \text{ if } \gamma_{n+1}=\gamma_n e^{i\theta_{k}}. 
        \end{equation*}
We say that the sequence $(s_n)$ given by $s_n=s_\Sigma(n)$ is the \em{constant sequence} of $(\gamma_n)$. Notice that we allow $c_l=c_k$ for $l \neq k$.
\end{definition} 

	 \begin{theorem}
	\label{th:maintheorem}
	Let $S=\{\theta_0, \theta_1,\dots,\theta_{m-1}\}$ be a generator set and $\Delta$ be the revolving group generated by S.
	For $\alpha\in\mathbb{C}$ with $|\alpha|<1$ and $c_0,c_1,\dots,c_{m-1}\in\mathbb{C}$, define $X_{\alpha,S}$ as follows.
        \begin{equation}
            X_{\alpha,S}=\left\{\sum_{n=1}^\infty\alpha^{n-1}s_n \gamma_n \::\:\Sigma=(\gamma_n)\in W^{\Delta},\:s_n=s_{\Sigma}(n)\right\}.
        \end{equation}
				
				Then $X_{\alpha,S}$ is the union of $|\Delta|$ many rotated copies of the self-similar attractor $T_{\alpha, S}$:
        \begin{equation*}
            X_{\alpha,S}=\bigcup_{\gamma\in\Delta}\gamma\cdot T_{\alpha,S},
        \end{equation*}
				where $T_{\alpha, S}$ is generated by the IFS:
				\begin{equation*}
\begin{cases}
 \psi_0(z)=\alpha z+c_0,& \\
 \psi_k(z)=(\alpha e^{i \theta_k}) z+c_k, \qquad k=1, 2, \dots, m-1.
\end{cases}
\end{equation*}
    \end{theorem}

		\begin{proof}[Proof.]~\\
		For a given generator set $S=\{\theta_0, \theta_1,\dots,\theta_{m-1}\}$, set $\alpha_{k}=\alpha e^{i\theta_{k}}$ for $k=0, 1, \dots, m-1$ in Lemma~\ref{lemma:coding expression}. Note that $\alpha_{0}=\alpha$ since $\theta_0=0$. Then it follows from \eqref{Attractor} that the self-similar set $T_{\alpha, S}$ has the following representation. 
    \begin{equation*}
        T_{\alpha, S}=\left\{\sum_{n=1}^{\infty}c_{x_n}\prod_{k=0}^{m-1}(\alpha e^{i\theta_k})^{I_k(x,n-1)}\::\:x=(x_n)\in\{0,1,\dots,m-1\}^{\NN}\right\}.
    \end{equation*}
		Observe that $\sum_{k=0}^{m-1}{I_k(x,n-1)}=n-1$ and  
	\begin{equation*}
	\prod_{k=0}^{m-1}e^{i\theta_k \cdot I_k(x,n-1)}=e^{i \sum_{k=0}^{m-1}I_k(x,n-1) \cdot \theta_k}=e^{i\sum_{j=1}^{n-1}\theta_{x_j}},
    \end{equation*}
Thus, 
    \begin{equation}
		\label{eq:T_aS}
        T_{\alpha, S}=\left\{\sum_{n=1}^{\infty}\alpha^{n-1}c_{x_n} e^{i \sum_{j=1}^{n-1}\theta_{x_j}}\::\: (x_n) \in\{0,1,\dots,m-1\}^{\NN}\right\}.
    \end{equation}

Define $X_{1,\alpha,S}$ given by:
        \begin{equation*}
        \label{eq:Attractor}
            X_{1,\alpha,S}=\left\{\sum_{n=1}^\infty\alpha^{n-1} s_n \gamma_n\::\:\Sigma=(\gamma_n)\in W^{\Delta}, \gamma_1=1, \:s_n=s_{\Sigma}(n)\right\}.
        \end{equation*} 
				
We prove that $X_{1,\alpha,S}=T_{\alpha,S}$. Let $\gamma_n=e^{i\sum_{j=1}^{n-1}\theta_{x_j}}$. Then it is clear that there exists a coding: 
\begin{equation*}
        \gamma_1=1, \qquad \gamma_{n+1}=\gamma_n \cdot e^{i \theta_{x_n}} 
\end{equation*}
for $n=1, 2, \dots$. Therefore, $(\gamma_n)$ satisfies the DRC with the restriction that the first term is $1$.

Let $s_n=c_{x_n}$. It is clear that $(s_n)$ is the constant sequence of $(\gamma_n)$. 
\medskip

Conversely, take a point $z \in X_{1,\alpha,S}$, which is generated by a sequence $(\gamma_n)_{n \in \NN}$. Construct a sequence $(x_n)\in\{0,1,\dots,m-1\}^{\NN}$ by  
\begin{equation*}
x_n=k, \qquad \mbox{if } \gamma_{n+1}=\gamma_n e^{i \theta_k}, 
\end{equation*}
for $k=0,1,2,\dots (m-1)$. Notice that $\theta_k$ is a unique revolving angle for each $k$ so that $(x_n)$ is determined uniquely by $(\gamma_n)$.   

Recall that for a given generator set $S$, the revolving group generated by $S$ has $\lcm (|p_1|, \cdots, |p_{m-1}|)$ many elements where $p_i$ is the denominator of $\theta_i$ in reduced form. Thus, 

\begin{equation*}
 X_{\alpha,S}=\bigcup_{\gamma \in \Delta}\gamma \cdot X_{1,\alpha,S}=\bigcup_{\gamma \in \Delta}\gamma \cdot T_{\alpha,S}.
\end{equation*}
 \end{proof}  

   \begin{figure}
	\begin{center}
		\includegraphics[height=4.5cm,width=4.5cm, bb=100 100 1000 1000]{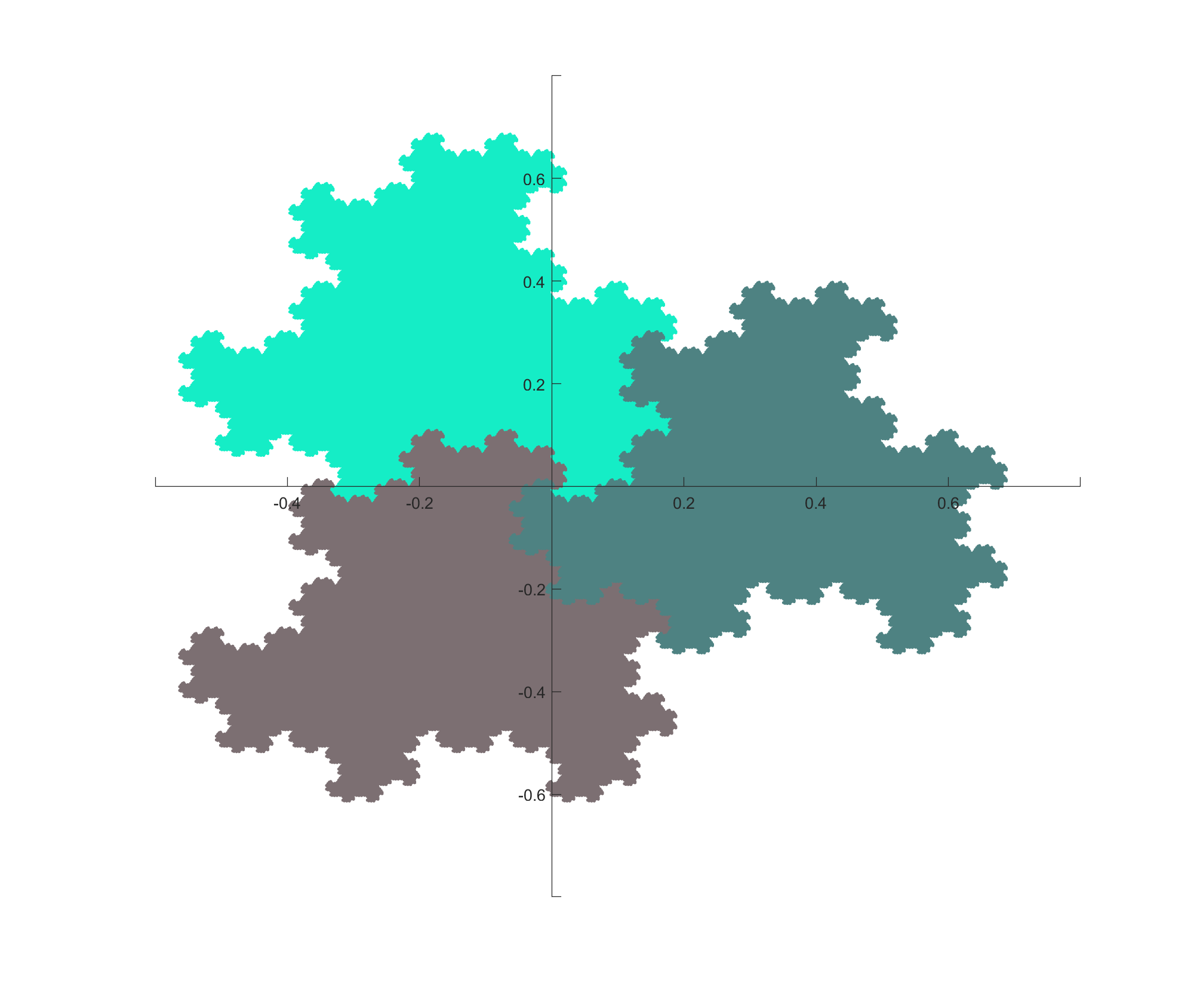}
		\end{center}
	\vspace{1cm}
	
    \includegraphics[height=4.5cm,width=4.5cm, bb=100 100 1000 1000]{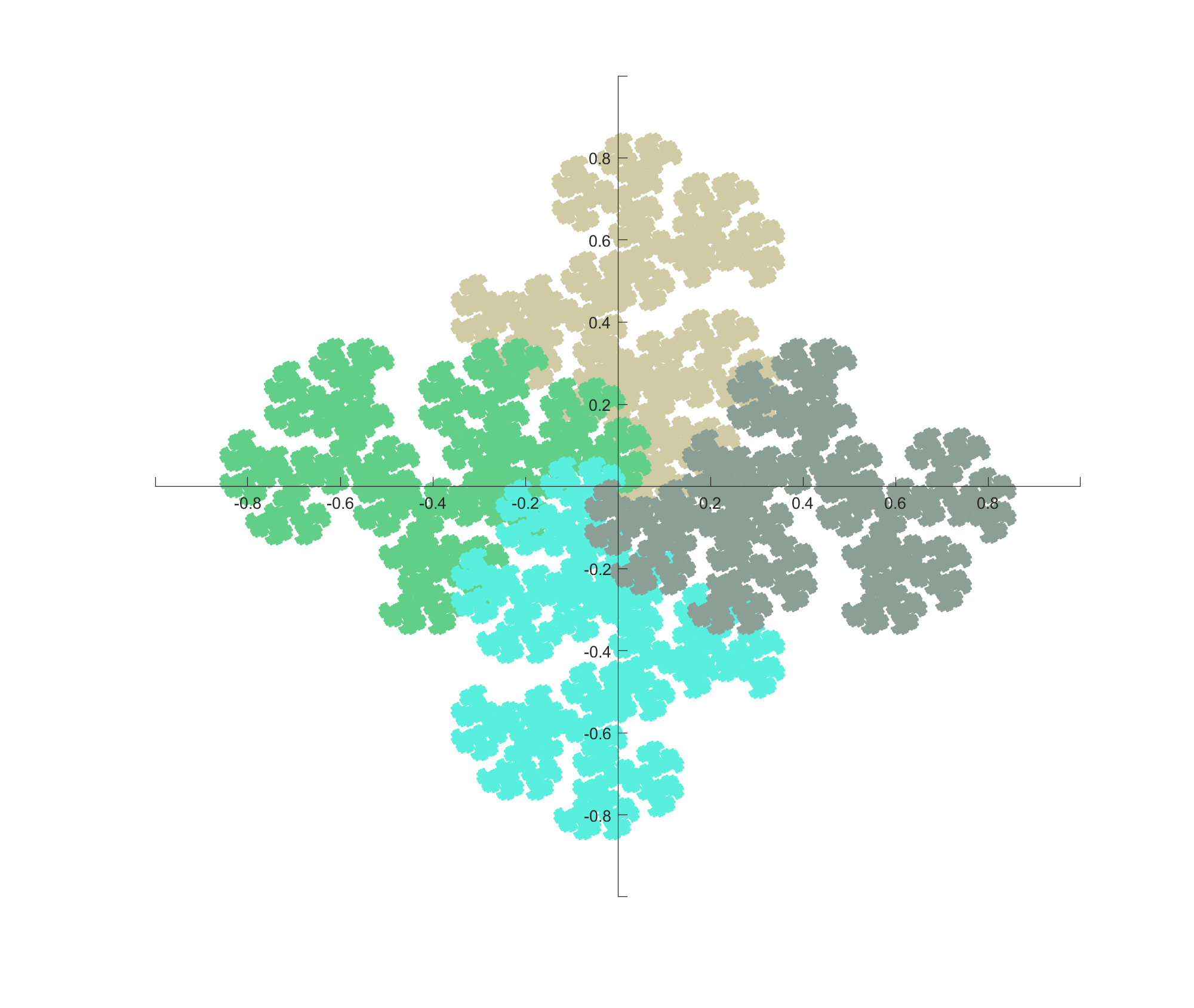}
    \qquad \qquad \qquad
    \includegraphics[height=4.5cm,width=4.5cm, bb=100 100 1000 1000]{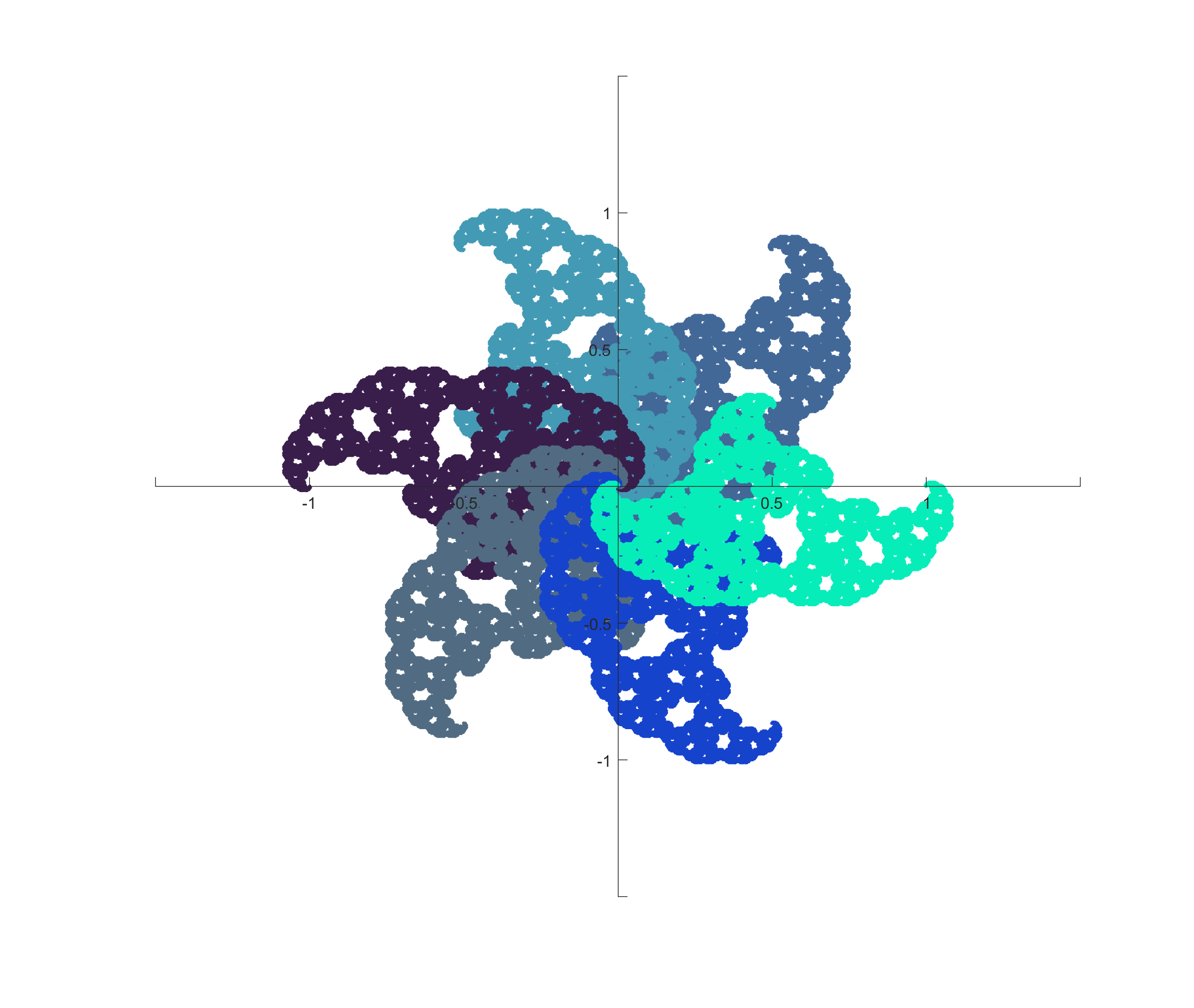}
    \caption{Examples of $X_{\alpha,S}$: $(m, c_0, c_1, c_2, \alpha)=(3, 0, \alpha, \bar{\alpha}, \frac12-\frac{\sqrt{3}i}{6})$.	
		$(\theta_1, \theta_2)=(\pi/3, -2\pi/3)$: a union of 3 fudgeflake (top),  
		$(\theta_1, \theta_2)=(\pi/2, -\pi/2)$: (bottom left), 
		$(\theta_1, \theta_2)=(\pi/3, -\pi/3)$: (bottom right).
		}
    \label{figure3} 
    \end{figure}

\begin{example}
{\rm 
Mandelbrot introduced a tiling fractal, known as {\it the fudgeflake} in his classic book (see page 72 in ~\cite{Mandelbrot-1982}). {\it The fudgeflake} is a self-similar attractor generated by three similar contractions:
\begin{equation}
\label{eq:IFS3}
\begin{cases}
 \psi_1(z)=\alpha z,& \\
 \psi_2(z)=(\alpha e^{i \theta_1}) z+ \alpha, & \\
 \psi_3(z)=(\alpha e^{i \theta_2}) z + \bar{\alpha},
\end{cases}
\end{equation}
where $\alpha=\frac12-\frac{\sqrt{3}i}{6}$, $\theta_1=\pi/3$ and $\theta_2=-2\pi/3$. For more details. see page 22-23 in Edger's book~\cite{Edgar-1990}.

It is clear that Theorem~\ref{th:maintheorem} includes this particular case. Figure \ref{figure3} shows several examples of $X_{\alpha,S}$; in particular, the self-similar attractors of \eqref{eq:IFS3}. 

{\it Terdragon} is another famous tiling self-similar attractor generated by \eqref{eq:IFS3} if $\alpha=\frac12-\frac{\sqrt{3}i}{6}$, $\theta_1=2\pi/3$ and $\theta_2=0$.(See page 163 in~\cite{Edgar-1990}). However, notice that Theorem~\ref{th:maintheorem} excludes this case since $\theta_2=0$.} 
\end{example}

\begin{example}{\rm 
The Heighway dragon and the Twindragon are famous self-similar attractors, not generated by \eqref{eq:IFS1} but by a slightly different pair of two similar contractions:
\begin{equation}
\label{eq:IFS2} 
\begin{cases}
 \phi_0(z)=\alpha z,& \\
 \phi_1(z)=(\alpha e^{i \theta}) z+1.
\end{cases}
\end{equation}
In particular, if $\alpha=(1+i)/2$ and $\theta=\pi/2$, then the IFS generates the Heighway dragon. If $\alpha=(1+i)/2$ and $\theta=\pi$, then  the Twindragon is generated~\cite{Edgar-1990}.

\medskip 

Let $S=\{0, \theta\}, c_0=0$ and $c_1=1$ in Theorem~\ref{th:maintheorem}. Since 
\begin{equation*}
		s_n=s_{\Sigma}(n)=
\begin{cases}
 0,& \mbox{ if } \gamma_{n+1}=\gamma_n,\\
 1, & \mbox{ if } \gamma_{n+1}=\gamma_n e^{i \theta},
\end{cases}
\end{equation*}
$X_{\alpha, S}$ can also be parametrized by generalized revolving sequences $(\delta_n)$.  
\begin{equation}
\label{eq:Heighway}
  X_{\alpha,S}=\left\{ \sum_{n=0}^{\infty} \delta_{n}\alpha^{n}: (\delta_{n})\in W_{\theta} \right\}.
\end{equation}
Note that the only difference between \eqref{eq:Heighway} and \eqref{eq:X1} is whether the sum begins from $n=0$ or $n=1$. 
It is not surprising. Let $H_{\alpha,\theta}=\phi_0(H_{\alpha,\theta}) \cup \phi_1(H_{\alpha,\theta})$ be the self-similar set generated by the IFS \eqref{eq:IFS2}. Notice that 
$\psi_0(z)=\alpha \phi_0(z/\alpha)$ and $\psi_1(z)=\alpha \phi_1(z/\alpha)$ in \eqref{eq:IFS1} and \eqref{eq:IFS2}. Thus, it is easy to see  that $H_{\alpha, \theta}=M_{\alpha,\theta}/\alpha$.
} 
\end{example}

    \begin{figure}
    \begin{center}
		\includegraphics[height=4.5cm,width=4.5cm, bb=100 100 1000 1000]{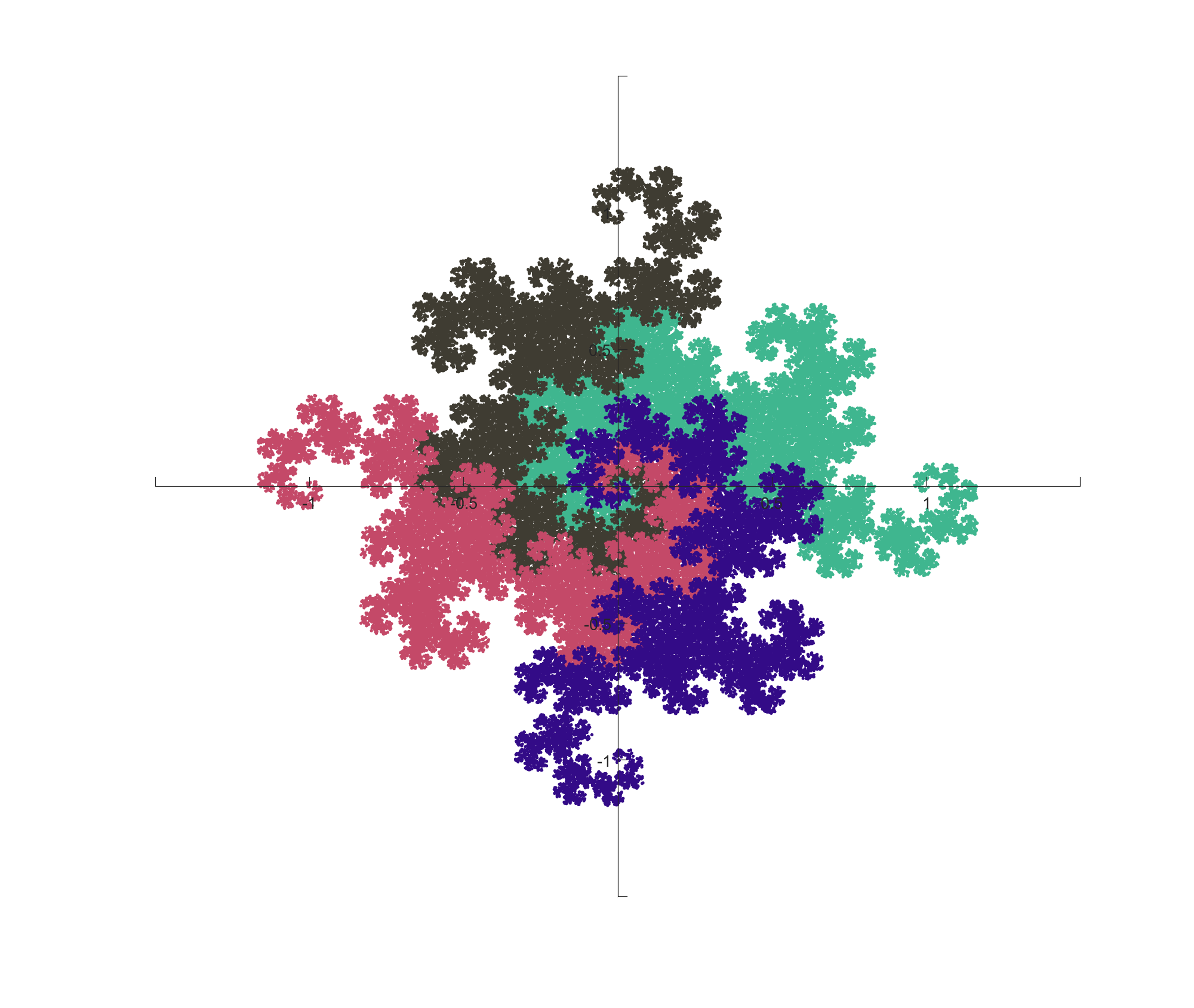} 
		\end{center}
		\caption{$X_{\alpha,S}$: $(m, c_0, c_1, \alpha, \theta)=(2, 0, 1, \frac{1+i}{2}, \frac{\pi}{2})$: a union of 4 Heighway dragons}
    \label{figure4} 
    \end{figure}


\begin{remark}
{\rm 
Theorem~\ref{theorem-Kawamura-Allen} is a special case of Theorem~\ref{th:maintheorem}. 
Let $S=\{0, \theta_1\}, c_0=0$ and $c_1=\alpha$ so that  
\begin{equation*}
		s_n=
\begin{cases}
 0,& \mbox{ if } \gamma_{n+1}=\gamma_n,\\
 \alpha, & \mbox{ if } \gamma_{n+1}=\gamma_n e^{i \theta_1}.
\end{cases}
\end{equation*}

Observe that $\delta_n=\alpha^{-1}\gamma_n s_n$ since (i) $\delta_n=0$ if $\gamma_{n+1}=\gamma_n$ and (ii) $\delta_n=\gamma_n$ if $\gamma_{n+1} =\gamma_n e^{i \theta_1}$. Therefore,  
\begin{equation*}
\left\{\sum_{n=1}^\infty\alpha^{n-1}\gamma_n s_n\::\:\Sigma=(\gamma_n)\in W^{\Delta},\:s_n=s_{\Sigma}(n)\right\}=\left\{\sum_{n=1}^{\infty} \delta_{n}\alpha^{n}: (\delta_n)\in W_{\theta} \right\},
\end{equation*}
where $W^{\Delta}$ is the set of all $\Delta$-revolving sequences, and $W_{\theta}$ is the set of all generalized revolving sequences.
}
\end{remark}

\section{Remark}
\label{remark}

First, we define a slight modification of {\it $\Delta$-revolving sequences} as follows.

    \begin{definition}[$\Delta_0$-Revolving Sequence]
		Let $\Delta_0:=\Delta\cup\{0\}$. We say that a sequence $(\delta_n)\in\Delta_{0}^\mathbb{N}$ satisfies the \textbf{$\Delta_0$-Revolving Condition (DZRC)} if
        \begin{enumerate}
					\item $\delta_1$ is free to choose in $\Delta_0$
					\item If $\delta_1=\delta_2=\dots=\delta_n=0$, then $\delta_{n+1}$ is free to choose in $\Delta_0$. 
          \item Otherwise, $\delta_{n+1}=0$ or $\delta_{n+1}=\delta_{j_0(n)} e^{i\theta_k}$ for some $\theta_k \in S \backslash \{0\}$,
        \end{enumerate}
        where $j_0(n)=max\{j \leq n:\delta_j \neq 0\}$.
\end{definition}

\begin{example}
If $S=\{0, \pi,\frac{2\pi}{3}\}$, $\Delta_0= \{0, 1,e^{i(\pi)},e^{i(\frac{2\pi}{3})},e^{i(\frac{4\pi}{3})},e^{i(\pi+\frac{2\pi}{3})},e^{i(\pi+\frac{4\pi}{3})}\}$, and, for examples, 
				\begin{equation}
				\label{example:delta-0}
				(\delta_n)=(1, e^{i\pi}, 0, e^{i(\pi+\frac{2\pi}{3})}, 0, 0,  e^{i(\frac{2\pi}{3})}, \cdots )
				\end{equation}
				is a $\Delta_0$-revolving sequence.
    \end{example} 
		
		We denote the set of all $\Delta_0$-revolving sequences with generator set $S$ by $W^{\Delta_0}$.
		
Compare examples: \eqref{example:delta} and \eqref{example:delta-0}. Observe that $\Delta_0$-revolving sequences are a more natural extension of the notation of generalized revolving sequences than $\Delta$-revolving sequences. In fact, $\Delta_0$-revolving sequences are past-dependent.

\medskip
Now, a question arises naturally: for a given $\alpha \in\CC$ such that $|\alpha|<1$ and a generator set $S$, define  
$$X^{*}_{\alpha,S}:=\left\{ \sum_{n=1}^{\infty} \delta_{n}\alpha^{n}: (\delta_n)\in W^{\Delta_0} \right\}.$$ 
It is not hard to imagine that $X^{*}_{\alpha, S}$ is a union of self-similar sets; however, it is not immediately clear which IFS generates these self-similar sets.

\begin{corollary}
 $X^{*}_{\alpha,S}$ is the union of $|\Delta|$ many rotated copies of the self-similar attractor $T^{*}_{\alpha, S}$:
        \begin{equation*}
            X^{*}_{\alpha,S}=\bigcup_{\gamma\in\Delta}\gamma\cdot T^{*}_{\alpha,S},
        \end{equation*}
				where $T^{*}_{\alpha, S}$ is generated by the IFS:
				\begin{equation*}
\begin{cases}
 \psi_0(z)=\alpha z & \\
 \psi_k(z)=(\alpha e^{i \theta_k}) z+\alpha, \qquad k=1, 2, \dots, m-1.
\end{cases}
\end{equation*}
\end{corollary}

\begin{remark}
{\rm In general, self-similar attractors generated by 
\begin{equation*}
\begin{cases}
 \psi_0(z)=\alpha z & \\
 \psi_k(z)=(\alpha e^{i \theta_k}) z+c_k, \qquad k=1, 2, \dots, m-1.
\end{cases}
\end{equation*}
cannot be parametrized by $\Delta_0$-revolving sequences. Consider \eqref{eq:T_aS} in the proof of Theorem~\ref{th:maintheorem}.
Since $c_0=0$, let $\delta_n=\rI(x_n) e^{i\sum_{j=1}^{n-1}\theta_{x_j}}$, where 
\begin{equation*}
		\rI(x_n)=
\begin{cases}
 0, & \mbox{ if } x_n=0, \\
 1, & \mbox{ if } x_n \neq 0.
\end{cases}
\end{equation*}
Then it is clear that there exists a $\Delta_0$-revolving sequence $(\delta_n)$ given by 
\begin{equation*}
		\delta_n=
\begin{cases}
 0, & \mbox{ if } x_n=0,\\
 \delta_{j_0(n-1)}e^{i \theta_{x_{j_0(n-1)}}} , & \mbox{ if $j_0(n-1)$ exists}, \\
 1, & \mbox{ if  $j_0(n-1)$ does not exist},
\end{cases}
\end{equation*}
where $j_0(n)=\max\{j \leq n: x_j \neq 0\}$. 
However, notice that there is an issue to define $(x_n)$ from $(\delta_n)$, if $(\delta_n)$ is a sequence whose terms are eventually all $0$.
} 
\end{remark}

\section*{Acknowledgment}
The first author would like to express her gratitude to the late Prof.~S.~Ito, who introduced one of his results in \cite{Ito-1987}. He inspired and encouraged many female mathematicians to study problems related to fractal geometry and number theory.   
Lastly, We greatly appreciate Prof.~P.~Allaart for his helpful comments and suggestions in preparing this paper.

\end{document}